\documentclass[11pt,reqno]{amsart}

\usepackage{amsmath,amssymb,amsthm}
\usepackage{bm}
\usepackage{natbib}
\usepackage{graphicx}
\usepackage{here}
\usepackage{multirow}

\makeatletter

\@addtoreset{equation}{section}
\makeatletter

\newtheorem{theorem}{Theorem}[section]

\newtheorem{lemma}{Lemma}[section]
\newtheorem{proposition}{Proposition}[section]

\theoremstyle{definition}

\newtheorem{remark}{Remark}[section]

\renewcommand{\hat}{\widehat}

\DeclareMathOperator{\Var}{Var}
\DeclareMathOperator{\Cov}{Cov}

\DeclareMathOperator{\Card}{Card}

\begin{document}

\title[]{On nonparametric inference for spatial regression models under domain expanding and infill asymptotics}
\thanks{
}

\author[D. Kurisu]{Daisuke Kurisu}

\date{This version: \today}


\address[D. Kurisu]{
Tokyo Institute of Technology\\
2-12-1 Ookayama, Meguro-ku, Tokyo 152-8552, Japan.
}
\email{kurisu.d.aa@m.titech.ac.jp}

\begin{abstract}
In this paper, we develop nonparametric inference on spatial regression models as an extension of \cite{LuTj14}, which develops nonparametric inference on density functions of stationary spatial processes under domain expanding and infill (DEI) asymptotics. In particular, we derive multivariate central limit theorems of mean and variance functions of nonparametric spatial regression models. Built upon those results, we propose a method to construct confidence bands for mean and variance functions. 
\medskip

\end{abstract}

\maketitle

\section{Introduction}

Recently, a considerable interest has been paid on statistical inference of spatial regression models for geostatistical data analysis in some economic and scientific fields such as spatial econometrics, ecology, and seismology. 
In this paper, we consider the following spatial regression model. 
\begin{align}\label{NSR}
Y(s_{j}) = \mu(X(s_{j})) + \sigma(X(s_{j}))V(s_{j}),\ j = 1,\hdots, n,
\end{align}
where $Y(s)$ and $X(s)$ is a scalar for each $s \in \mathbb{R}^{2}$, $\{V(s_{j})\}_{j=1}^{n}$ are independent and identically distributed (i.i.d.) random variables, and $\{Y(s_{j}), X(s_{j})\}_{j = 1}^{n}$ are observations. The functions $\mu(x): \mathbb{R} \to \mathbb{R}$ and $\sigma(x): \mathbb{R} \to [0,\infty)$ are mean and conditional variance functions, respectively. 
\cite{LuTj14} proposed a scheme of domain-expanding and infill (DEI) asymptotics and studied nonaprametric density estimation for spatial data under this sampling scheme. In this paper, we are extending their results to spatial regression models in a relatively simple setting under their sampling scheme. 
As mentioned in their paper, in many applications, the DEI sampling scheme is natural because of physical constraints measurement stations cannot usually be put on a regular grid in space. Precisely, Let $\|\cdot\|$ be the Euclidean norm on $\mathbb{R}^{2}$ and define
\begin{align*}
\delta_{j,n} &= \min\{\|s_{k} - s_{j}\||1 \leq k \leq n, k \neq j\},\ \Delta_{j,n} = \max\{\|s_{k} - s_{j}\||1 \leq k \leq n, k \neq j\}.
\end{align*}
Then we assume that $\delta_{n} = \max_{1 \leq j \leq n}\delta_{n,j} \to 0,\ \Delta_{n} = \min_{1 \leq j \leq n}\Delta_{n,j} \to \infty$ as $n \to \infty$. There are some papers whose sampling schemes are related to the DEI asymptotics. \cite{HaPa94} investigated nonparametric estimation of spatial covariance function based on observations generated by a probability distribution. \cite{MaYa09} work with a similar sampling scheme and focus on the nonparametric and parametric estimation of the spectral density.

Recent contributions in the literature of statistical inference on spatial (or random field) models include \cite{JePr09, JePr12}, \cite{Ma11}, and \cite{MaVoWu13} which investigate limit theorems for the statistical inference on spatial process observed on (irregularly spaced) lattice. 
In the literature on semi-parametric spatial (or spatio-temporal) regression models, we mention \cite{GaLuTj06}, \cite{LuStTjYa09}, \cite{YaHoWeZu14} and \cite{AlJiLuZh17} as recent key references. We can also find an overview for recent developments on semi-parametric spatial models in \cite{Ro08}. 
\cite{Ro11} and \cite{Je12} study nonparametric inference on spatial regression models under a dependence structure which is different from our mixing-type conditions. They derive central limit theorems of mean functions and the former discusses an application of the results to spatial data observed on lattice. 
We also refer to \cite{HaLuTr04}, \cite{MaSt08}, \cite{HaLuYu09}, \cite{Li16} and \cite{MaSeOu18} which study nonparametric inference and estimation on mean function of spatial regression models based on random fields on lattice. However, those papers do not derive limit theorems of variance functions. 

The goal of this paper is to derive multivariate central limit theorems of the mean and variance function of the model (\ref{NSR}). From a technical point of view, we cannot use small-block and large-block argument due to \cite{Be26} which is a well known tool for nonparametric inference for regularly observed time series data since we work with DEI sampling scheme. Although we can apply the blocking argument for regularly spaced data, there is no practical guidance for constructing small and large blocks under the asymptotic framework that the distance between observations goes to $0$. Therefore, to avoid the problem, we use another approach due to \cite{Bo82} which is based on the convergence of characteristic functions of the estimators. As a result, this paper contributes to the literature on nonparametric inference for spatial regression models, and to the best of our knowledge, this is the first paper to derive limit theorems for the mean and variance functions of the model (\ref{NSR}) under the DEI asymptotics. 

The rest of the paper is organized as follows. In Section \ref{Sec2}, we give conditions to derive limit theorems given in this paper. In Section \ref{Sec3}, we give  multivariate central limit theorems of the mean and variance functions and propose a method to construct confidence bands for the estimators of those functions on finite intervals included in $\mathbb{R}$. We also propose a data-driven method for bandwidth selection and 
report simulation results to study finite sample performance of the central limit theorems and proposed confidence bands in Section \ref{Sec4}. 
All proofs are collected in Appendix A. 

\subsection{Notations}
For any non-empty set $T$ and any (complex-valued) function $f$ on $T$, let $\|f\|_{T} = \sup_{t \in T}|f(t)|$, and for $T = \mathbb{R}$, let $\|f\|_{L^{p}} = (\int_{\mathbb{R}}|f(x)|^{p}dx)^{1/p}$ for $p >0$. 

\section{Assumptions}\label{Sec2}

In this section we summarize assumptions used in the proof of limit theorems given in Section 3 for the sake of convenience. 

(A1): \textbf{Assumption on spatial process}
\begin{itemize}
\item[(i)] $\{X(s), s \in \mathbb{R}^{2}\}$ is a strictly stationary spatial process, satisfying the $\alpha$-mixing property that there exist a function $\varphi$ such that $\varphi(t) \downarrow 0$ as $t \to \infty$, and a function $\psi: \mathbb{N}^{2} \to [0,\infty)$ symmetric and increasing in each of its two arguments, such that
\begin{align}
\alpha(\mathcal{B}(\mathcal{S}'), \mathcal{B}(\mathcal{S}'')) &= \sup_{A \in \mathcal{B}(\mathcal{S}'), B \in \mathcal{B}(\mathcal{S}'')}|P(A \cap B) - P(A)P(B)| \nonumber  \\
&\leq \psi(\Card(\mathcal{S}'), \Card(\mathcal{S}''))\varphi(d(\mathcal{S}', \mathcal{S}'')), \label{MixRF}
\end{align}
where $\mathcal{S}', \mathcal{S}'' \subset \mathbb{R}^{2}$, $\mathcal{B}(\mathcal{S})$ be the Borel $\sigma$-field generated by $\{X_{s}, s \in \mathcal{S}\}$, and $d(\mathcal{S}', \mathcal{S}'') = \min\{\|s_{k} - s_{j}\|| s_{k} \in \mathcal{S}', s_{j} \in \mathcal{S}''\}$ for each $\mathcal{S}'$ and $\mathcal{S}''$.  

\item[(ii)] For some constant $\gamma>\max\{1, 2\kappa/(2 + \kappa)\}$ and some $\kappa>0$, 

\noindent 
$\limsup_{m \to \infty}m^{4 + 3\gamma}\sum_{j=m}^{\infty}j\varphi^{\kappa/(2 + \kappa)}(j) <\infty$.

\item[(iii)] $n \psi(1,n)\varphi(c_{n}) \to 0$ as $n \to \infty$, where $c_{n} = \{\delta_{n}^{2}h^{\kappa/(2 + \kappa)}\}^{-1/\gamma}$. 
\end{itemize} 

(A2): \textbf{Assumption on bandwidths and sampling scheme}

As $n \to \infty$, 
\begin{itemize}
\item[(i)] $b_{n}, h_{n} \to 0$.

\item[(ii)] $n(b_{n} + h_{n}) \to \infty$, $\liminf_{n \to \infty}n(b_{n}^{2} + h_{n}^{2})>0$, and $nb_{n}^{5} \to 0$, $nh_{n}^{5} \to 0$. 

\item[(iii)] $\delta_{n}^{-(2 + 2/\gamma)}(b_{n} + h_{n})^{1-2\kappa/\{(2 + \kappa)\gamma\}} \to 0$.
\end{itemize}

(A3): \textbf{Assumption on kernel function}
\begin{itemize}
\item[(i)] The kernel function $K$ is bounded, symmetric, and has a  bounded support. Let $c_{K} = \int_{\mathbb{R}}z^{2}K(z)dz$. 
\end{itemize}

(A4): \textbf{Assumption on regression models}

Let $f$, $f_{j,k}$, $f_{j,k,\ell}$, and $f_{j,k, \ell,m}$ be density functions of $X(s)$, $(X(s_{j}), X(s_{k}))$, $(X(s_{j}), X(s_{k}), X(s_{\ell}))$ and $(X(s_{j}), X(s_{k}), X(s_{\ell}), X(s_{m}))$ with pairwise distinct $j,k,\ell,m$, respectively. For some $\epsilon>0$ and $U \subset \mathbb{R}^{d}$, $U^{\epsilon}$ denotes an $\epsilon$-enlargement of $U$, that is, $U^{\epsilon} := \{x: \|x - y\| <\epsilon, y \in U\}$. 
\begin{itemize}
\item[(i)] $\inf_{x \in I^{\epsilon}}f(x)>0$ and $\inf_{x \in I^{\epsilon}}\sigma(x)>0$ for some compact set $I \subset \mathbb{R}$ and some $\epsilon>0$ and $I \subset \mathbb{R}$. 


\item[(ii)] $\mu \in C^{4}(I^{\epsilon})$, $\sigma \in C^{2}(I^{\epsilon})$, and $f$, $f_{j,k}$, $f_{j,k,\ell}$, and $f_{j,k,\ell,m}$ are bounded uniformly with respect to pairwise distinct $j,k,\ell,m$ where $I^{\otimes d} = \underbrace{I \times \hdots \times I}_{d}$.


Here, for $S \subset \mathbb{R}$,  $C^{p}(S) = \{f: \|f^{(k)}\|_{S}<\infty, k = 0,1,\hdots, p\}$ is the set of functions having bounded derivatives on $S$ up to order $p$, and $C^{0}(S)$ is the set of continuous functions on $S$. 

\item[(iii)] $E[V(s_{1})] = 0$, $E[V^{2}(s_{1})] = 1$, and $E[|V(s_{1})|^{8 + 4\kappa}]<\infty$. Here, $\kappa>0$ is the constant which appear in Assumption (A1) (ii). 

\item[(iv)] $\{X(s): s \in \mathbb{R}^{2}\}$ and $\{V(s): s \in \mathbb{R}^{2}\}$ are independent. 
\end{itemize}

\begin{remark}
The random field $\{Y(s_{j}), X(s_{j})\}$ is called strongly mixing if the condition (\ref{MixRF}) holds with $\psi \equiv 1$. The same or similar conditions are used in \cite{HaLuTr04}, 
and \cite{LuTj14}. The condition can be seen as an extension of strong mixing conditions for continuous-time stochastic processes and time series models. It is known that many stochastic processes and (nonlinear) time series models are strongly mixing. 
The conditions (\ref{MixRF}) and (iii) in Assumption (A1) are 
satisfied by many spatial processes. We refer to \cite{Ro85} and \cite{Gu87} for detailed discussion on strong mixing conditions for random fields. 
\end{remark}


\section{Multivariate central limit theorems of the mean and variance functions}\label{Sec3}

In this section we give central limit theorem of the marginal density, mean, and variance functions of the nonparametric spatial regression model (\ref{NSR}). 
\subsection{Limit theorems for mean and variance functions}

Let $K$ be a kernel function with $\int_{\mathbb{R}}K(x)dx = 1$, $b_{n}$ and $h_{n}$ are bandwidth with $b_{n}, h_{n} \to 0$ and $n(b_{n} + h_{n}) \to \infty$ as $n \to \infty$. We estimate $\mu(x)$ and $\sigma^{2}(x)$ by 
\begin{align*}
\widehat{\mu}_{b_{n}}(x) &= {1 \over nb_{n}\widehat{f}_{X}(x)}\sum_{j=1}^{n}Y(s_{j})K\left({x-X(s_{j}) \over b_{n}}\right),\\
\widehat{\sigma}^{2}_{h_{n}}(x) &=  {1 \over nh_{n}\widetilde{f}_{X}(x)}\sum_{j=1}^{n}\left\{Y(s_{j}) - \widehat{\mu}_{b_{n}}^{\ast}(X(s_{j}))\right\}^{2}K\left({x-X(s_{j}) \over h_{n}}\right)
\end{align*}
where
\begin{align*}
\widehat{f}_{X}(x) &= {1 \over nb_{n}}\sum_{j=1}^{n}K\left({x- X(s_{j}) \over b_{n}}\right),\ 
\widetilde{f}_{X}(x) = {1 \over nh_{n}}\sum_{j=1}^{n}K\left({x- X(s_{j}) \over h_{n}}\right),\ \widehat{\mu}_{b_{n}}^{\ast}(x) &= 2\widehat{\mu}_{b_{n}}(x) - \widehat{\mu}_{\sqrt{2}b_{n}}(x).
\end{align*}

\begin{remark}
The estimator $\widehat{\mu}_{b_{n}}^{\ast}(x)$ is a jackknife version of $\widehat{\mu}_{b_{n}}(x)$. Although, from a theoretical point of view, it is sufficient to assume $\mu \in C^{2}(I^{\epsilon})$ in Assumption (A4) (ii) to derive limit theorems on $\mu$, we use this estimator instead of $\widehat{\mu}_{b_{n}}(x)$ in the definition of $\widehat{\sigma}^{2}_{h_{n}}(x)$ to ignore the effects of its asymptotic bias and for the improvement of its finite sample performance. In fact, under Assumption (A3) and (A4) (four-times continuous differentiability of $\mu$ on $I^{\epsilon}$), we can show that $\widehat{\mu}_{b_{n}}(x) - \mu(x) = O_{P}(b_{n}^{2})$ and $\widehat{\mu}_{b_{n}}^{\ast}(x) - \mu(x) = O_{P}(b_{n}^{4})$ for each $x \in \mathbb{R}$. 
\end{remark}

Now we give limit theorems for $\mu$ and $\sigma$. First we give a multivariate extension of Theorem 1 in \cite{LuTj14}. 
\begin{proposition}\label{P1}
Under Assumptions (A1), (A2), (A3) and (A4), for $-\infty<x_{1}<x_{2}<\cdots < x_{N}<\infty$, we have that
\[
\sqrt{{nb_{n} \over\|K\|_{L^{2}}^{2}}}\left({\widehat{f}(x_{1}) - f(x_{1})  \over \sqrt{f(x_{1})}},\hdots, {\widehat{f}(x_{N}) - f(x_{N})  \over \sqrt{f(x_{N})}} \right)^{\top} \stackrel{d}{\to} N(0,I_{N}). 
\]
where $I_{N}$ is the $N \times N$ identity matrix. 
\end{proposition}


Next we give a general limit theorem for nonparametric spatial regression models. The following theorem are used to prove multivariate central limit theorems of $\hat{\mu}_{b_{n}}$ and $\hat{\sigma}^{2}_{h_{n}}$. 
\begin{proposition}\label{CLT}
Assume that $G$ and $H$ are functions such that 
\begin{itemize}
\item[(i)] $G \in C^{0}(\{x\}^{\epsilon})$ for some $\epsilon>0$.  
\item[(ii)] $H: \mathbb{R} \to \mathbb{R}$ with $E[H(V(s_{1}))] = 0$, $E[H^{2}(V(s_{1}))]>0$ and $E[|H(V(s_{1}))|^{4 + 2\kappa}]<\infty$. 
\end{itemize}
Define
\[
U_{n} = {1 \over \sqrt{nh_{n}}}\sum_{j=1}^{n}G(X(s_{j}))H(V(s_{j}))K\left({x - X(s_{j}) \over h_{n}}\right).
\] 
Then, under Assumptions (A1), (A2), (A3) and (A4), for $x \in I$ we have that 

$U_{n} \stackrel{d}{\to} N(0, f(x)G^{2}(x)V(H)\|K\|_{L^{2}}^{2})$ as $n \to \infty$, where $V(H) = E[H^{2}(V(s_{1}))]$. 
\end{proposition}

By using Proposition \ref{CLT}, we can finally derive the following two theorems. 
\begin{theorem}\label{P2}
Under Assumptions (A1), (A2), (A3) and (A4), for $-\infty<x_{1}<x_{2}<\cdots < x_{N}<\infty$, we have that
\begin{align*}
&\sqrt{{nb_{n} \over\|K\|_{L^{2}}^{2}}}\left({\sqrt{f(x_{1})}(\widehat{\mu}_{b_{n}}(x_{1}) - \mu(x_{1}))  \over \sqrt{\sigma^{2}(x_{1})}},\hdots, {\sqrt{f(x_{N})}(\widehat{\mu}_{b_{n}}(x_{N}) - \mu(x_{N}))  \over \sqrt{\sigma^{2}(x_{N})}} \right)^{\top} \stackrel{d}{\to} N(0,I_{N}). 
\end{align*}
\end{theorem}


\begin{theorem}\label{P3}
Under Assumptions (A1), (A2), (A3) and (A4), and $E[V^{4}(s_{1})]>1$, for $-\infty<x_{1}<x_{2}<\cdots < x_{N}<\infty$, we have that
\begin{align*}
&\sqrt{{nh_{n} \over V_{4}\|K\|_{L^{2}}^{2}}}\left({\sqrt{f(x_{1})}(\widehat{\sigma}^{2}_{h_{n}}(x_{1}) - \sigma^{2}(x_{1}))  \over \sqrt{\sigma^{4}(x_{1})}},\hdots, {\sqrt{f(x_{N})}(\widehat{\sigma}^{2}_{h_{n}}(x_{N}) - \sigma^{2}(x_{N}))  \over \sqrt{\sigma^{4}(x_{N})}} \right)^{\top} \stackrel{d}{\to} N(0,I_{N}). 
\end{align*}
where $V_{4} = E[V^{4}(s_{1})]-1$. 
\end{theorem}


\begin{remark}
The mixing condition could be relaxed to a more general near epoch dependence (at least for random fields observed at (irregularly spaced) lattice points). For example, we refer to \cite{LuLi07} and Li, Lu and Linton(2012). They study nonparametric inference and estimation of time series regression models respectively. However, their analysis is based on regularly spaced time series and the assumption is essential in those papers. Recently, there are new techniques for Gaussian approximation of time series based on strong approximation (e.g. \cite{LiLi09}) or combinations of Slepian's smart path interpolation (\cite{Rol11}) to the solution of Stein' s partial deferential equation, Stein's leave-one(-block)-out method (\cite{St86}) and other analytical techniques  (e.g. see \cite{ZhWu17} and \cite{ZhCh17} under different physical dependence, and \cite{ChChKa13} under $\beta$-mixing sequences). If we work with domain-expanding but not infill asymptotics ($\lim_{n \to \infty}\Delta_{n} = \infty$ and $\liminf_{n \to \infty}\delta_{n}>0$), we would able to relax our mixing condition and also provide a method to construct (asymptotically) uniform confidence bands as results of high-dimensional extensions of our theorems to the case that number of design points in an interval increases as the sample size $n$ goes to infinity (i.e. $x_{j} \in I, j=1,\hdots,N$, $N = N_{n} \to \infty$ as $n \to \infty$) by using techniques in those papers (see \cite{Ku18} for time series case). However, if we work with DEI asymptotics, to achieve such results would need careful treatment of the dependence among observations and the author believes that it requires substantial work. 
\end{remark}

\subsection{Confidence bands for mean and variance functions}

Based on the multivariate central limit theorems in the previous section, we propose a method to construct confidence bands for mean and variance functions on a finite interval $I$ included in $\mathbb{R}$. 
We estimate $V_{4}$ by 
\[
\widehat{V}_{4} = {\sum_{j = 1}^{n}\widehat{V}^{4}(s_{j})1\{X(s_{j}) \in I\} \over \sum_{j=1}^{n}1\{X(s_{j}) \in I\}} -1,\ \text{where}\ \widehat{V}(s_{j}) = {Y(s_{j}) - \widehat{\mu}^{\ast}_{b_{n}}(X(s_{j})) \over \widehat{\sigma}_{h_{n}}(X(s_{j}))}
\]
instead of the naive estimator $n^{-1}\sum_{j=1}^{n}\widehat{V}^{4}(s_{j})-1$ to improve finite sample performance. 
Let $\xi_{1},\hdots, \xi_{N}$ be i.i.d. standard normal random variables, and let $q_{\tau}$ satisfy $P\left(\max_{1 \leq j \leq N}|\xi_{j}| > q_{\tau}\right) = \tau$ for $\tau \in (0,1)$.
Then, $\widehat{C}_{F_{1}}(x_{j}) = \left[\widehat{F}_{1}(x_{j}) \pm {\widehat{F}_{2}(x_{j}) \over \sqrt{nh_{n}}}q_{\tau}\right],\ j = 1,\hdots, N$ 
are joint asymptotic $100(1-\tau)$\% confidence intervals of $f$, $\mu$, and $\sigma^{2}$ when $(\hat{F}_{1}, \hat{F}_{2}) = \left(\hat{f}, \sqrt{\hat{f}}\|K\|_{L^{2}}\right),\ \left(\hat{\mu}, \hat{\sigma}\|K\|_{L^{2}}/\sqrt{\hat{f}}\right)$, and $\left(\hat{\sigma}^{2}, \hat{\sigma}^{2}\|K\|_{L^{2}}\sqrt{\hat{V}_{4}/\hat{f}}\right)$, 
respectively. 

\section{Simulations}\label{Sec4}
In this section we present simulation results to see the finite-sample performance of the central limit theorems and proposed confidence bands in Section \ref{Sec3}. 

\subsection{Simulation framework}

To generate the locations irregularly positioned in $\mathbb{R}^{2}$, first we set a lattice $(u_{j}, v_{k})$ with $u_{j} = u_{0} + (j-1)\times 0.3$ and $v_{k} = v_{0} + (k-1)\times 0.3$ for $j,k = 1,\hdots, n$ where $u_{0} = 0.3$ and $v_{0} = 0.6$. Next we select $n$ locations randomly from the lattice as the irregular locations $(u_{j_{\ell}},v_{k_{\ell}})$ with $1 \leq j_{\ell}, k_{\ell} \leq n$, $\ell = 1,\hdots n$ and set $s_{\ell} = (u_{j_{\ell}}, v_{k_{\ell}})$. As a data generating process, we consider the following spatial moving average process. 
\begin{align}\label{SMA}
X(s_{j}) &= \sum_{\ell = -1}^{1}\sum_{m = -1}^{1}a_{\ell,m}Z_{\ell,m},\ s_{j} = (u_{\ell}, v_{m}),\ j=1,\hdots, n,
\end{align}
where $Z_{\ell,m}$ are independent and identically distributed standard normal random variables, $a_{\ell,m}$ is the $(\ell +2,m+2)$-component of the matrix 
\begin{align*}
A &= 
\left(
\begin{array}{ccc}
1/5 & 2/5 & -4/5 \\
-3/5 & -2/5 & -1/5 \\
-1/5 & 2/5 & -3/5
\end{array} 
\right).
\end{align*}
We also consider $\mu(x) = 0.1 + 0.3x$ and $\sigma^2(x) = 0.2 + 0.05x + 0.3x^2$ as the mean and variance functions respectively, and use i.i.d. standard Gaussian random variables as noise variables $\{V(s_{j})\}_{j=1}^{n}$. In our simulation study, we use the Epanechnikov kernel $K(x) = {3 \over 4}(1-x^{2})1\{|x| \leq 1\}$ and set the sample size $n$ as 750. Note that Assumptions (A1) on the spatial process $X$ is satisfied from the definition of the spatial moving average process (\ref{SMA}). 

\subsection{Bandwidth selection}
Now we discuss bandwidth selection for the construction of confidence bands on a finite interval $I$. Let $I$ be a finite interval, $x_{1}<\cdots < x_{J} $ be design points with $x_{j} \in I$ for $j = 1,\hdots, J$, and let $0<b_{1}<\cdots<b_{L}$ and $0<h_{1}<\cdots<h_{L}$ be grids of bandwidths. We use a data-driven method which is similar to that proposed in \cite{Ku18}. 
From a theoretical point of view, we have to choose bandwidths that are of smaller order than the optimal rate for estimation under the loss functions (or a ``discretized version'' of $L^{\infty}$-distances) $\max_{1 \leq j \leq J}| \hat{\mu}_{b_{\ell}}(x_{j}) - \mu(x_{j})|$ and $\max_{1 \leq j \leq J}| \hat{\sigma}^{2}_{h_{\ell}}(x_{j}) - \sigma^{2}(x_{j})|$ for our confidence bands to work. At the same time, choosing a too small bandwidth results in a too wide confidence band. Therefore, we should choose a bandwidth ``slightly'' smaller than the optimal one that minimizes those loss functions. 
We employ the following rule for bandwidth selection of $\hat{\mu}_{b_{n}}$. We also choose a bandwidth of $\hat{\sigma}^{2}_{h_{n}}$ in a similar manner.  
\begin{enumerate}
\item Set a pilot bandwidth $b^{P} >0$ and make a list of candidate bandwidths $b_{\ell} = \ell b^{P}/L$ for $\ell=1,\dots,L$. 
\item Choose the smallest bandwidth $b_{\ell} \ (\ell \geq 2)$ such that the adjacent value $\max_{1 \leq j \leq J}| \hat{\mu}_{b_\ell}(x_{j}) - \hat{\mu}_{b_{\ell-1}}(x_{j})|$ is smaller than $\tau \times \min \{ \max_{1 \leq j \leq J}| \hat{\mu}_{b_{\ell}}(x_{j}) - \hat{\mu}_{b_{\ell-1}}(x_{j}) | : \ell=2,\dots,L \}$ for some $\tau > 1$. 
\end{enumerate}
For the bandwidth selection of $\hat{\sigma}^{2}_{h_{n}}$, we first choose a bandwidth ($\hat{b}_{n}$) of $\hat{\mu}^{\ast}_{b_{n}}$ by the proposed rule. Then plug $\hat{\mu}^{\ast}_{\hat{b}_{n}}$ into $\hat{\sigma}^{2}_{h_{n}}$ and we finally choose a bandwidth $\hat{h}_{n}$ by the proposed rule. In our simulation study, we choose $b^{P} = h^{P} = 1, L = 20$, and $\tau= 2$. This rule would choose a bandwidth slightly smaller than a bandwidth 
which is intuitively the optimal bandwidth for the estimation of $\mu$ and $\sigma^{2}$ as long as the threshold value $\tau$ is reasonably chosen.

Figure \ref{fig:A1} depicts five realizations of the loss functions $\max_{1 \leq j \leq J}|\hat{\mu}_{b_{n}}(x_{j}) - \mu(x_{j})|$ (left) and $\max_{1 \leq j \leq J}|\hat{\sigma}^{2}_{h_{n}}(x_{j}) - \sigma^{2}(x_{j})|$(right) with different bandwidth values and $x_{j} = -0.5 + (j-1) \times 0.1$, $j = 1,\hdots, 11$. 
Figure \ref{fig:A2} depicts five realizations of the loss function $\max_{2 \leq j \leq 11}|\hat{\mu}_{b_{\ell}}(x_{j}) - \hat{\mu}_{b_{\ell-1}}(x_{j})|$ (left) and $\max_{1 \leq j \leq 11}|\hat{\sigma}^{2}_{h_{\ell}}(x_{j}) - \hat{\sigma}^{2}_{h_{\ell-1}}(x_{j})|$ (right). It is observed that the shape of $\max_{1 \leq j \leq 11}| \hat{\mu}_{b_{\ell}}(x_{j}) - \hat{\mu}_{b_{\ell-1}}(x_{j})|$  partly mimics that of $\max_{1 \leq j \leq 11}| \hat{\mu}_{b_{\ell}}(x_{j}) - \mu(x_{j})|$. The same thing can be said about $\max_{1 \leq j \leq 11}| \hat{\sigma}^{2}_{h_{\ell}}(x_{j}) - \hat{\sigma}^{2}_{h_{\ell-1}}(x_{j})|$.  By using the proposed rule and the visual information of Figure \ref{fig:A1}, we set $b_{n} = h_{n} = 0.5$ to draw Figures \ref{fig:A4} and \ref{fig:A5}. 


\begin{remark}
In practice, it is also recommended to make use of visual information on how $\max_{1\leq j \leq J}| \hat{\mu}_{b_{\ell}}(x_{j}) - \hat{\mu}_{b_{\ell-1}}(x_{j})|$ and $\max_{1\leq j \leq J}| \hat{\sigma}^{2}_{h_{\ell}}(x_{j}) - \hat{\sigma}^{2}_{h_{\ell-1}}(x_{j})|$ behave as $\ell$ increases when determining the bandwidth. 
\end{remark}

\begin{figure}[H]
  \begin{center}
    \begin{tabular}{cc}

      \begin{minipage}{0.5\hsize}
        \begin{center}
          \includegraphics[clip, width=4cm]{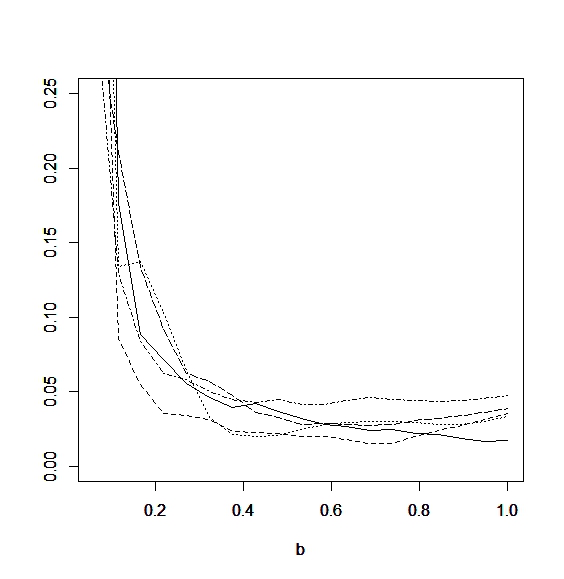}
        \end{center}
      \end{minipage}

      \begin{minipage}{0.5\hsize}
        \begin{center}
          \includegraphics[clip, width=4cm]{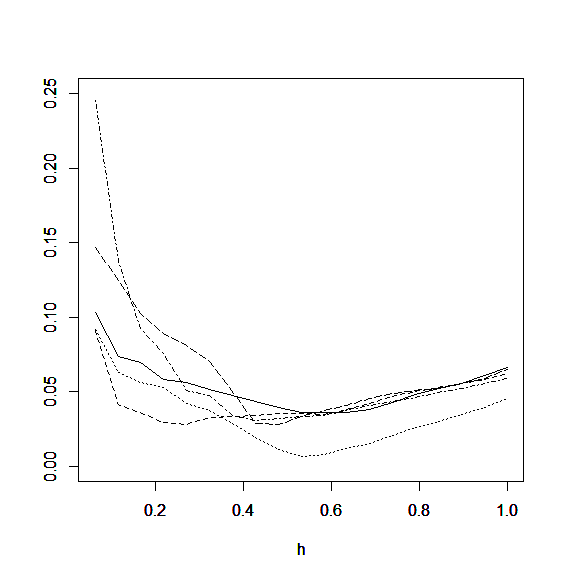}
        \end{center}
      \end{minipage}

    \end{tabular}
    \caption{Discrete $L^{\infty}$-distance between the true function $\mu$ and estimates $\hat{\mu}_{b_{n}}$ (left), and between the true function $\sigma^{2}$ and estimates $\hat{\sigma}^{2}_{h_{n}}$ (right) for different bandwidth values. We set $n = 750$, $I = [-0.5, 0,5]$, and $x_{j} = -0.5 + (j-1) \times 0.1$, $j = 1,\hdots, 11$. \label{fig:A1}}   
  \end{center}
\end{figure}

\begin{figure}[H]
  \begin{center}
    \begin{tabular}{cc}

      \begin{minipage}{0.5\hsize}
        \begin{center}
          \includegraphics[clip, width=4cm]{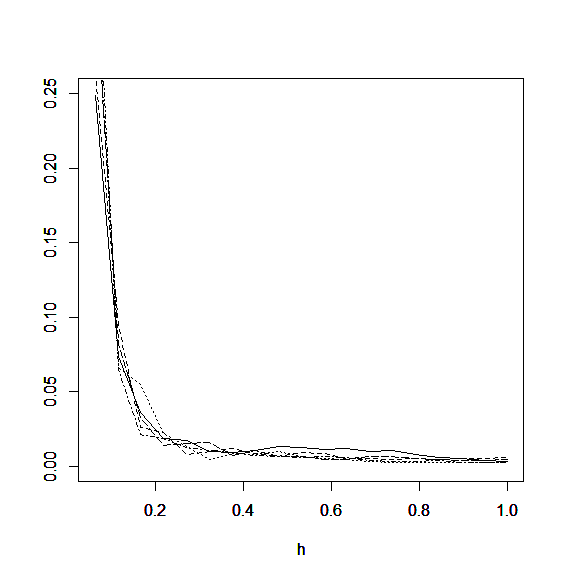}
        \end{center}
      \end{minipage}

      \begin{minipage}{0.5\hsize}
        \begin{center}
          \includegraphics[clip, width=4cm]{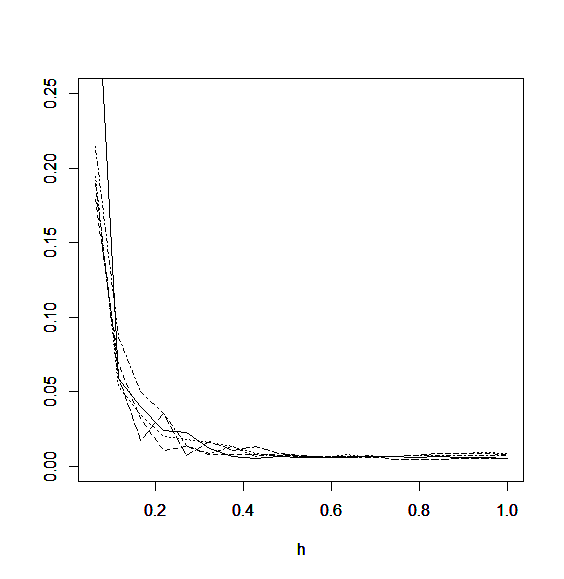}
        \end{center}
      \end{minipage}

    \end{tabular}
    \caption{Discrete $L^{\infty}$-distance between the estimates $\hat{\mu}_{b_{n}}$ (left), and between the estimates $\hat{\sigma}^{2}_{h_{n}}$ (right) for different bandwidth values. We set $n = 750$, $I = [-0.5, 0,5]$, and $x_{j} = -0.5 + (j-1) \times 0.1$, $j = 1,\hdots, 11$. \label{fig:A2}}   
  \end{center}
\end{figure}

\begin{figure}[H]
  \begin{center}
     \includegraphics[clip, width=4cm]{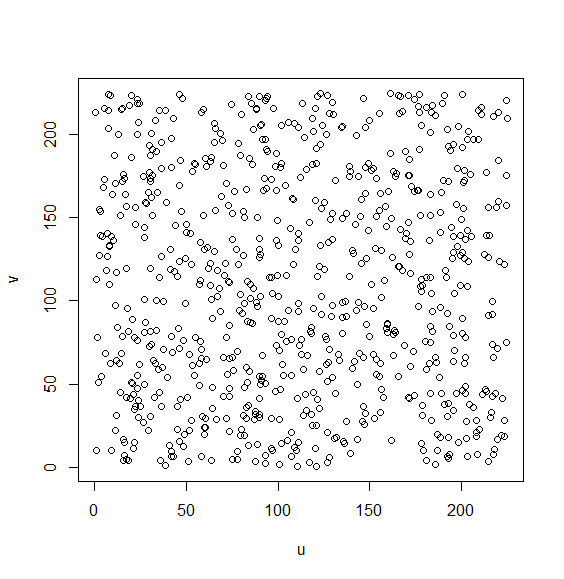}
        \caption{Simulated observation points 
sampled from the lattice $(u_{j},v_{k})$ with $u_{j} = 0.3 + 0.3(j-1)$ and $v_{k} = 0.6 + 0.3(k-1)$, $j,k = 1,\hdots, 750$. }
    \label{fig:A3}
  \end{center}
\end{figure}

Figure \ref{fig:A3} shows a simulated irregularly spaced observation points. Figures \ref{fig:A4} and \ref{fig:A5} show normalized empirical distributions of $\hat{\mu}_{b_{n}}(x)$ and $\hat{\sigma}_{h_{n}}^{2}(x)$ with $b_{n} = h_{n} = 0.5$
at $x = -0.25$(left), $x = 0$(center), and $x = 0.25$(right). The red line is the density of a standard normal distribution. The number of Monte Carlo iteration is 250 for each case. From these figures we can find the central limit theorems for each design point (Theorems \ref{P2} and \ref{P3}) hold true. 

Figure \ref{fig:A6} shows $85\%$(dark gray), $95\%$(gray), and $99\%$(light gray) confidence bands of $\mu$(left) and $\sigma^{2}$(right) on $I = [-0.5, 0.5]$. We set $x_{j} = -0.5 + (j-1)\times 0.1$, $j = 1,\hdots, 11$ as design points. 

We also investigated finite sample properties of $\widehat{\mu}_{b_{n}}$ and $\widehat{\mu}^{\ast}_{b_{n}}$ for different sample sizes. 
Figure \ref{fig:R1} depicts five realizations of the loss functions $\max_{1 \leq j \leq J}|\hat{\mu}_{b_{n}}(x_{j}) - \mu(x_{j})|$ (black line) and $\max_{1 \leq j \leq J}|\hat{\mu}^{\ast}_{b_{n}}(x_{j}) - \mu(x_{j})|$(red line) with different bandwidth values. We set $n = 750$(left), $1000$(center) and $1250$(right), and $x_{j} = -0.5 + (j-1) \times 0.1$, $j = 1,\hdots, 11$. We can find that red lines tend to be located below black lines. This implies that $\widehat{\mu}^{\ast}_{b_{n}}$ tends to have small bias compared with $\hat{\mu}_{b_{n}}$ as the sample size increases.

\begin{figure}[H]
  \begin{center}
    \begin{tabular}{cc}

      \begin{minipage}{0.33\hsize}
        \begin{center}
          \includegraphics[clip, width=3.8cm]{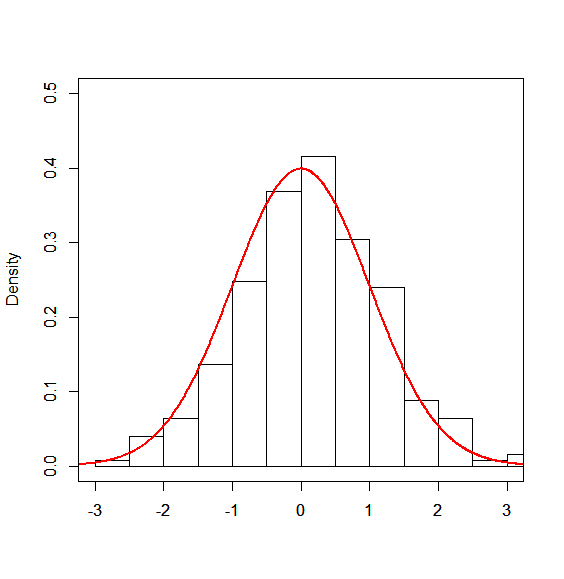}
        \end{center}
      \end{minipage}

      \begin{minipage}{0.33\hsize}
        \begin{center}
          \includegraphics[clip, width=3.8cm]{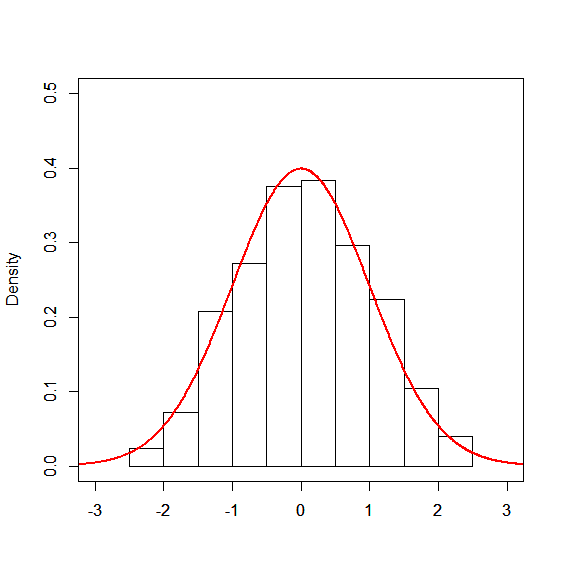}
        \end{center}
      \end{minipage}

      \begin{minipage}{0.33\hsize}
        \begin{center}
          \includegraphics[clip, width=3.8cm]{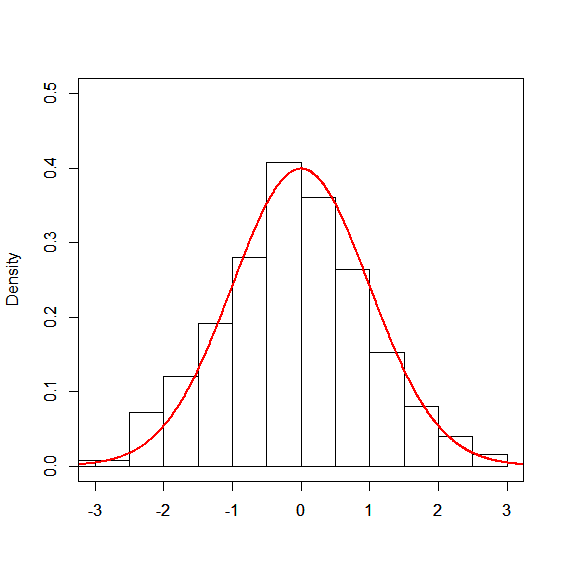}
        \end{center}
      \end{minipage}

    \end{tabular}
    \caption{Normalized empirical distributions of estimates $\hat{\mu}_{b_{n}}(x)$ at $x = -0.25$(left), $x=0$(center) and $x = 0.25$(right). The red line is the density of the standard normal distribution. We set $n = 750$ and the number of Monte Carlo iteration is 250.\label{fig:A4}}
  \end{center}
\end{figure}

\begin{figure}[H]
  \begin{center}
    \begin{tabular}{cc}

      \begin{minipage}{0.33\hsize}
        \begin{center}
          \includegraphics[clip, width=3.8cm]{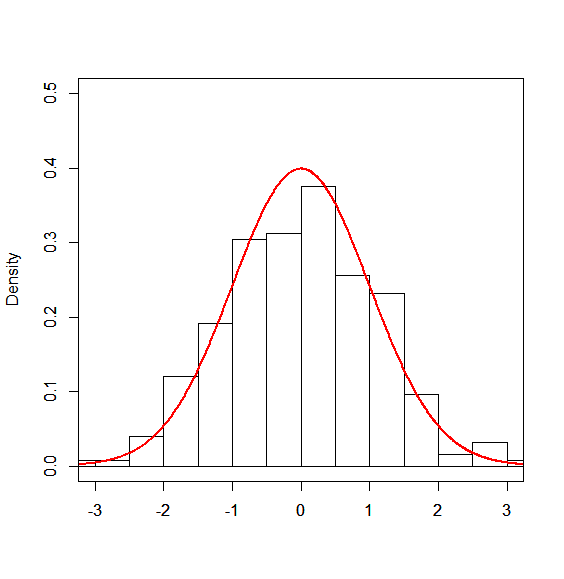}
        \end{center}
      \end{minipage}

      \begin{minipage}{0.33\hsize}
        \begin{center}
          \includegraphics[clip, width=3.8cm]{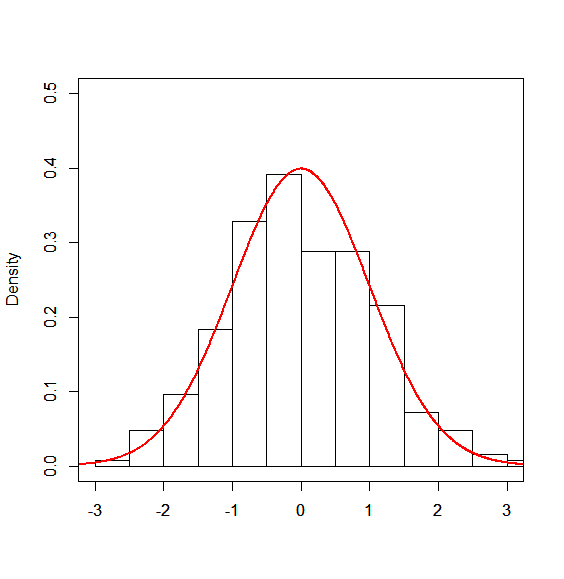}
        \end{center}
      \end{minipage}

      \begin{minipage}{0.33\hsize}
        \begin{center}
          \includegraphics[clip, width=3.8cm]{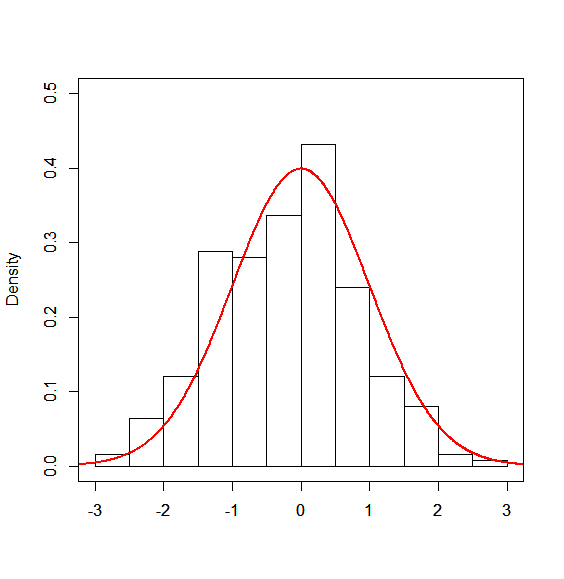}
        \end{center}
      \end{minipage}

    \end{tabular}
    \caption{Normalized empirical distributions of estimates $\hat{\sigma}^{2}_{h_{n}}(x)$ at $x = -0.25$(left), $x=0$(center) and $x = 0.25$(right). The red line is the density of the standard normal distribution. We set $n = 750$ and the number of Monte Carlo iteration is 250.\label{fig:A5}}
  \end{center}
\end{figure}

\begin{figure}[H]
  \begin{center}
    \begin{tabular}{cc}

      \begin{minipage}{0.5\hsize}
        \begin{center}
          \includegraphics[clip, width=4cm]{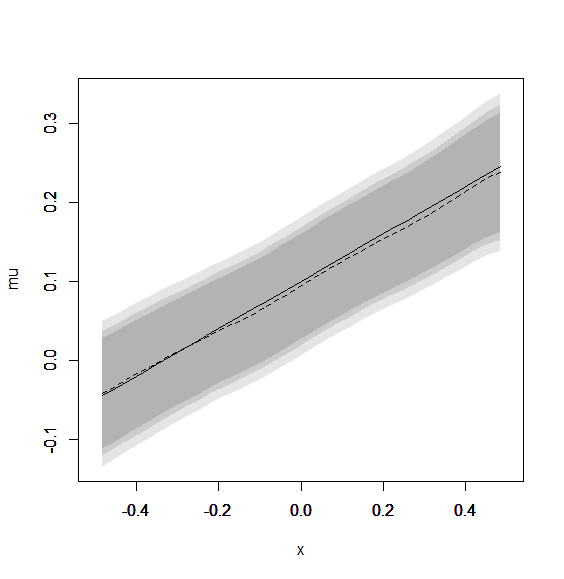}
        \end{center}
      \end{minipage}

      \begin{minipage}{0.5\hsize}
        \begin{center}
          \includegraphics[clip, width=4cm]{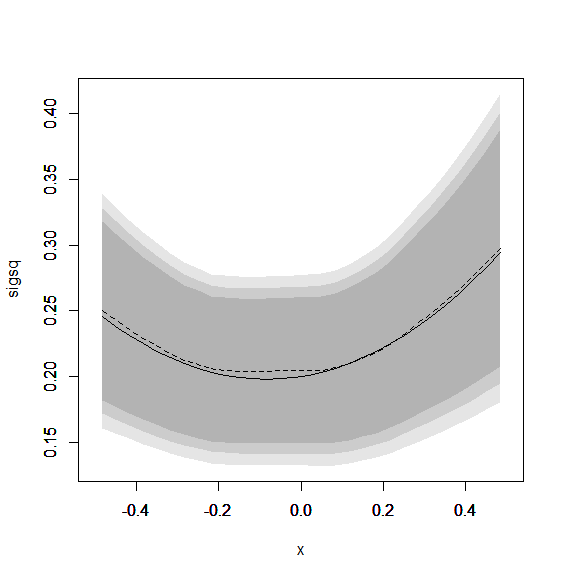}
        \end{center}
      \end{minipage}

    \end{tabular}
    \caption{ Estimates of $\mu$ (left) and $\sigma^{2}$ (right) together with $85\%$(dark gray), $95\%$(gray), and $99\%$(light gray) confidence bands. The solid lines correspond to the true functions. We set $n = 750$, $I = [-0.5,0.5]$, and $x_{j} = -0.5 + (j-1) \times 0.1$, $j = 1,\hdots, 11$. }
    \label{fig:A6}
  \end{center}
\end{figure}

\begin{figure}[H]
  \begin{center}
    \begin{tabular}{cc}

      \begin{minipage}{0.33\hsize}
        \begin{center}
          \includegraphics[clip, width=4cm]{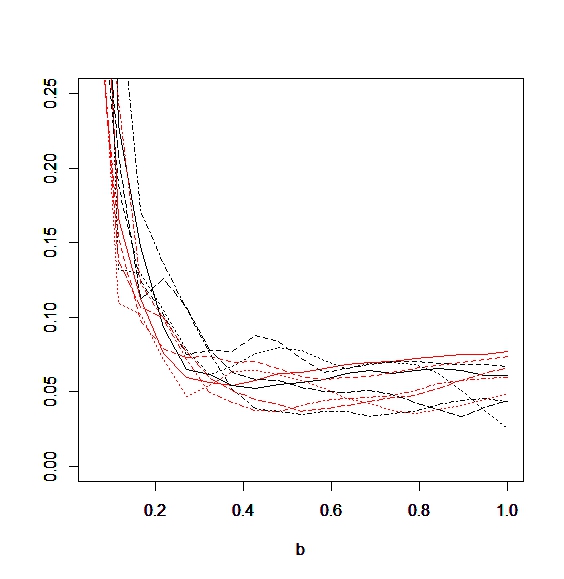}
        \end{center}
      \end{minipage}

      \begin{minipage}{0.33\hsize}
        \begin{center}
          \includegraphics[clip, width=4cm]{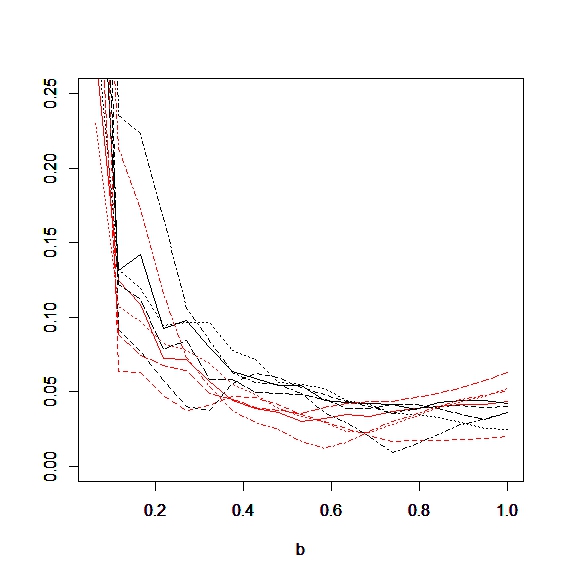}
        \end{center}
      \end{minipage}

      \begin{minipage}{0.33\hsize}
        \begin{center}
          \includegraphics[clip, width=4cm]{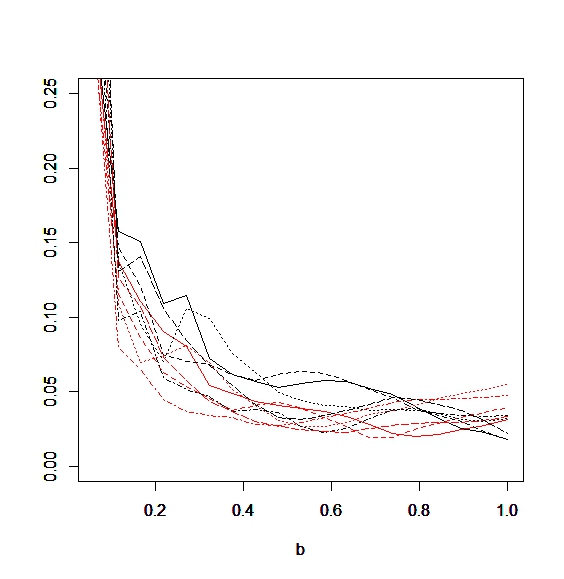}
        \end{center}
      \end{minipage}

    \end{tabular}
    \caption{Discrete $L^{\infty}$-distance between $\mu$ and estimates $\hat{\mu}_{b_{n}}$ (black line) and between $\mu$ and estimates $\hat{\mu}^{\ast}_{b_{n}}$ (red line) for different bandwidth values. We set $n = 750$(left), $1000$(center) and $1250$(right), $I = [-0.5, 0,5]$, and $x_{j} = -0.5 + (j-1) \times 0.1$, $j = 1,\hdots, 11$. \label{fig:R1}}
  \end{center}
\end{figure}
\section*{Acknowlegements}
This work is partially supported by Grant-in-Aid for Research Activity Start-up (18H05679) from the JSPS. 

\appendix

\section{Proofs}
\begin{proof}[Proof of Proposition \ref{P1}]
We already have pointwise central limit theorem of $\widehat{f}(x)$ by Theorem 1 in \cite{LuTj14}. Then a multivariate central limit theorem also holds true by Cram\'er-Wald device. Therefore, it suffices to show the asymptotic independence of estimators at different design points. For this, it is sufficient to show that covariances between estimators at different design points are asymptotically negligible. For $x_{1}<x_{2}$, we have that
\begin{align*}
nb_{n}\!\Cov(\widehat{f}(x_{1}), \widehat{f}(x_{2})) &\!=\! {1 \over nb_{n}}\sum_{j=1}^{n}\Cov(K((x_{1} - X(s_{j})/b_{n})), K((x_{2} - X(s_{j})/b_{n})))\\
&\quad + {2 \over nb_{n}}\!\sum_{i=1}^{n-1}\!\sum_{j = i+1}^{n-1}\!\!\Cov(K((x_{1} - X(s_{i})/b_{n})), K((x_{2} - X(s_{j})/b_{n}))) =: \!{1 \over nb_{n}}(V_{n,1} \!+\! V_{n,2}). 
\end{align*}
For $V_{n,1}$, by the assumption on $K$ and a change of variables, we have that 
\begin{align*}
\left|E\left[{1 \over hb_{n}}V_{n,1}\right]\right| &\leq \int K(z)K\left(z + {x_{2} - x_{1} \over b_{n}}\right)f(x_{1} - b_{n}z)dz + b_{n}\left(\int K(z)dz\right)^{2} = O(b_{n}) \to 0 
\end{align*}
as $n \to \infty$. We can also show that $(nb_{n})^{-1}V_{n,2} \to 0$ as $n \to \infty$ by the same argument of the proof of Theorem 1 in \cite{LuTj14}. Therefore, we complete the proof. 
\end{proof}

\begin{proof}[Proof of Proposition \ref{CLT}]
Let $i = \sqrt{-1}$. For the proof of the central limit theorem, it is sufficient to show the following conditions by Lemma 2 in \cite{Bo82}. 
\begin{itemize}
\item[(a)] $\sup_{n \geq 1}E[U_{n}^{2}]<\infty$ and (b) $\lim_{n \to \infty}E[(i\lambda - U_{n})e^{i\lambda U_{n}}] = 0$ for any $\lambda  \in \mathbb{R}$. 
\end{itemize}
The condition (a) immediately follows from the definition of $U_{n}$, since by a change of variables and the dominated convergence theorem, we have that 
\begin{align*}
nh_{n}\Var(U_{n}) &= {1 \over nh_{n}}\sum_{j=1}^{n}E\left[G^{2}(X(s_{j}))H^{2}(V(s_{j}))K^{2}\left((x - X(s_{j})) /h_{n}\right)\right] \\
&= E[H^{2}(V(s_{1}))]\int G^{2}(x - h_{n}z)K^{2}(z)f(x - h_{n}z)dz \to f(x)G^{2}(x)V(H)\|K\|_{L^{2}}^{2}\ \text{as}\ n \to \infty. 
\end{align*}
To show the condition (b), we use some technique in \cite{LuTj14}. 

(Step1): Define
\begin{align*}
Z_{n,j} = {1 \over nh_{n}}G(X(s_{j}))H(V(s_{j}))K\left({x- X(s_{j}) \over h_{n}}\right),
\end{align*} 
$U'_{n,j} = \sum_{\ell \in J_{n,j}(1)}Z_{n,\ell}$ where $J_{n,j}(p) := \{1 \leq \ell \leq n : d(\{s_{j}\}, \{s_{\ell}\}) \leq pc_{n}\}$, $a_{n} = \sum_{j=1}^{n}E[Z_{n,j}U'_{n,j}] = \Var(U_{n})(1 + o(1)) = O(1/nh_{n})$, and $U_{n,j} = a_{n}^{-1/2}U'_{n,j}$. 
We first give a lemma used in the proof of the condition (b). 
\begin{lemma}\label{Lem1}
We have that 
\begin{align*}
E[Z_{n,j}^{2}Z_{n,\ell}^{2}] &= 
\begin{cases}
O((n^{4}h_{n}^{3})^{-1}) & \text{if $j = \ell$}\\
O((n^{4}h_{n}^{2})^{-1}) & \text{if $j \neq \ell$} 
\end{cases},\\ 
\left|\Cov(Z_{n,j}Z_{n,\ell}, Z_{n,j'}Z_{n,\ell'})\right| &=
\begin{cases}
O\left({h^{2/(2 + \kappa)} \over n^{4}h_{n}^{4}}\right)\varphi^{\kappa/(2 + \kappa)}(d(\{s_{j}\}, \{s_{j'}\})) & \text{if $j = \ell$, $j' = \ell'$} \\
0 & \text{if $j \neq \ell$ or $j' \neq \ell'$} 
\end{cases}. 
\end{align*}
\end{lemma}
\begin{proof}[Proof of Lemma \ref{Lem1}]
We show the case when $j = \ell$ in the first result. The other case can be shown in the same way. If $j = \ell$, as $n \to \infty$, we have that
\begin{align*}
n^{4}h_{n}^{3}E[Z_{n,j}^{2}Z_{n,\ell}^{2}] &= {E[H^{4}(V(s_{1}))] \over h_{n}}\int G^{2}(y)K^{2}\left((x - y) / h_{n}\right)f(y)dy\\
&= E[H^{4}(V(s_{1}))]\int G^{2}(x - h_{n}z)K^{2}(z)f(x - h_{n}z)dz \to E[H^{4}(V(s_{1}))]f(x)G^{2}(x)\|K\|_{L^{2}}^{2}. 
\end{align*}
Now we show the second result. If $j = \ell$ and $j' = \ell'$, by Proposition 2.5 in \cite{FaYa03}, we have that 
\begin{align*}
\left|\Cov(Z_{n,j}Z_{n,\ell}, Z_{n,j'}Z_{n,\ell'})\right| &\leq \|Z_{n,j}^{2}\|_{L^{2 + \kappa}}\|Z_{n,j'}^{2}\|_{L^{2 +\kappa}}\alpha^{\kappa/(2 +\kappa)}(d(\{s_{j}\}, \{s_{j'}\}))\\
&\lesssim \|Z_{n,j}^{2}\|_{L^{2 + \kappa}}^{2}\varphi^{\kappa/(2 + \kappa)}(d(\{s_{j}\}, \{s_{j'}\})) \lesssim {h^{2/(2 + \kappa)} \over n^{4}h_{n}^{4}}\varphi^{\kappa/(2 + \kappa)}(d(\{s_{j}\}, \{s_{j'}\})). 
\end{align*}
Therefore, we complete the proof. 
\end{proof}
We decompose $E[(i\lambda - U_{n})e^{i\lambda U_{n}}]$ as follows. $E[(i\lambda - U_{n})e^{i\lambda U_{n}}] = A_{n,1} - A_{n,2} - A_{n,3}$
where 
\begin{align*}
A_{n,1} &= E\left(i\lambda e^{i\lambda U_{n}}\left(1 - {1 \over a_{n}}\sum_{j=1}^{n}Z_{n,j}U'_{n,j}\right)\right),\ A_{n,2} = E\left({1 \over \sqrt{a_{n}}}e^{i\lambda U_{n}}\sum_{j=1}^{n}Z_{n,j}(1 - i\lambda U_{n,j} - e^{-i\lambda U_{n,j}})\right), \\
A_{n,3} &= E\left({1 \over \sqrt{a_{n}}}\sum_{j=1}^{n}Z_{n,j}e^{i\lambda(U_{n} - U_{n,j})}\right).
\end{align*}

(Step 2): In this step, we show that $|A_{n,1}| \to 0$ as $n \to \infty$. 

Observe that $|A_{n,1}| \leq a_{n}^{-1}\left(\sum_{j=1}^{n}E\left(a_{n} - \sum_{j=1}^{n}Z_{n,j}U'_{n,j}\right)^{2}\right)^{1/2}$. 
We have that
\begin{align*}
E\left(a_{n} - \sum_{j=1}^{n}Z_{n,j}U'_{n,j}\right)^{2} &\!\!\!= E\left(\sum_{j=1}^{n}(Z_{n,j}U'_{n,j} - E[Z_{n,j}U'_{n,j}])\right)^{2}\\
&\!\!\!= \sum_{j=1}^{n}\sum_{k \in J_{n,j}(3)}\!\!\!\!\!\! \Cov(Z_{n,j}U'_{n,j}, Z_{n,k}U'_{n,k}) + \sum_{j=1}^{n}\sum_{k \in J_{n,j}^{c}(3)}\!\!\!\!\!\! \Cov(Z_{n,j}U'_{n,j}, Z_{n,k}U'_{n,k}) \\
&= : A_{n,11} + A_{n,12}. 
\end{align*}
By Lemma \ref{Lem1}, since 
\begin{align*}
|\Cov(Z_{n,j}U'_{n,j}, Z_{n,k}U'_{n,k})| & \leq E[Z_{n,j}^{2}|U'_{n,j}|^{2}] = \sum_{\ell \in J_{n,j}(1)}E[Z_{n,j}^{2}Z_{n,\ell}^{2}] + 2\sum_{ \ell < \ell', \ell, \ell' \in J_{n,j}(1)}E[Z_{n,j}^{2}Z_{n,\ell}Z_{n,\ell'}]\\
&\lesssim \left({c_{n} \over \delta_{n}}\right)^{2}\left({1 \over n^{4}h_{n}^{3}} +\left({c_{n} \over \delta_{n}}\right)^{2}{1 \over n^{4}h_{n}^{2}} \right) \lesssim \left({c_{n} \over \delta_{n}}\right)^{2}{1 \over n^{4}h_{n}^{3}}, 
\end{align*}
we have that $|A_{n,11}| \lesssim \left(c_{n}/\delta_{n}\right)^{4}(nh_{n})^{-3} \to 0$ as $n \to \infty$. Moreover, by Lemma \ref{Lem1}, we have that 
\begin{align*}
|A_{n,12}| &\leq \sum_{j=1}^{n}\sum_{j' \in J_{n,j}^{c}(3)}\left|\Cov(Z_{n,j}U'_{n,j}, Z_{n,j'}U'_{n,j'})\right| \leq \sum_{j=1}^{n}\sum_{j' \in J_{n,j}^{c}(3)}\sum_{\ell \in J_{n,j}(1)}\sum_{\ell' \in J_{n,j'}(1)}\!\!\!\!\!\left|\Cov(Z_{n,j}Z_{n,\ell}, Z_{n,j'}Z_{n,\ell'})\right|\\
&\lesssim \sum_{t =c_{n}}^{\infty}\sum_{j=1}^{n}\sum_{t \leq d(\{s_{j}\}, \{s_{j'}\}) \leq t+1}\left({c_{n} \over \delta_{n}}\right)^{4}{h^{2/(2 + \kappa)} \over n^{4}h_{n}^{4}}\varphi^{\kappa/(2 + \kappa)}(t) \lesssim \left({c_{n} \over \delta_{n}}\right)^{4}{h^{2/(2 + \kappa)} \over n^{3}h_{n}^{4}}\sum_{t = c_{n}}^{\infty}\left({t \over \delta_{n}^{2}}\right)\varphi^{\kappa/(2 + \kappa)}(t)\\
&\lesssim {h^{2/(2 + \kappa)} \over \delta_{n}^{6}n^{3}h_{n}^{4}}c_{n}^{4}\sum_{j=c_{n}}^{\infty}t\varphi^{\kappa/(2 + \kappa)}(t) = {1 \over n^{3}h_{n}^{(6 + \kappa)/(2 + \kappa)}}c_{n}^{4 + 3\gamma}\sum_{t = c_{n}}^{\infty}t\varphi^{\kappa/(2 + \kappa)}(t) \to 0\ \text{as}\ n \to \infty. 
\end{align*}

(Step 3): In this step, we show that $|A_{n,2}| \to 0$ as $n \to \infty$. Since $|1 - i\lambda U_{n,j} - e^{-i\lambda U_{n,j}}| \leq a_{n}^{-1}\lambda^{2}|U'_{n,j}|^{2}$,
we have that
\begin{align*}
|A_{n,2}| &\lesssim a_{n}^{-3/2}\sum_{j=1}^{n}E[|Z_{n,j}||U'_{n,j}|^{2}] \leq a_{n}^{-3/2}\sum_{j=1}^{n}\left(E[Z_{n,j}^{2}|U'_{n,j}|^{2}]\right)^{1/2}\left(E[|U'_{n,j}|^{2}]\right)^{1/2}\\
&\leq a_{n}^{-3/2}\left(\sum_{j=1}^{n}E[Z_{n,j}^{2}|U'_{n,j}|^{2}]\right)^{1/2}\left(\sum_{j=1}^{n}E[|U'_{n,j}|^{2}]\right)^{1/2}. 
\end{align*}
From a similar argument in (Step1), we have that $\sum_{j=1}^{n}E[Z_{n,j}^{2}|U'_{n,j}|^{2}] = O((nh_{n})^{-3})$. We also have that 
\begin{align*}
\sum_{j=1}^{n}E[|U'_{n,j}|^{2}] &\leq \sum_{j=1}^{n}E\left(\left|\sum_{\ell \in J_{n,j}(1)}Z_{n,j}\right|^{2}\right) \leq \sum_{j=1}^{n}\sum_{\ell \in J_{n,j}}\!\!\!E[Z_{n,\ell}^{2}] + \sum_{j=1}^{n}\sum_{\ell \neq \ell', \ell, \ell' \in J_{n,j}(1)}\!\!\!\!\!\!\!\!\!\!\!\!E[Z_{n,\ell}Z_{n,\ell'}] \lesssim \left({c_{n} \over \delta_{n}}\right)^{2}{1 \over nh_{n}}. 
\end{align*}
Therefore, we have that $|A_{n,2}| \lesssim a_{n}^{-3/2}(nh_{n})^{-3/2}\left({c_{n} \over \delta_{n}}\right){1 \over \sqrt{nh_{n}}} \lesssim \left({c_{n} \over \delta_{n}}\right){1 \over \sqrt{nh_{n}}} \to 0\ \text{as}\ n \to \infty$. 

(Step 4): In this step, we show that $|A_{n,3}| \to 0$ as $n \to \infty$. Since $E[Z_{n,j}] = 0$, by Theorem 5.1 in \cite{RoIo87}, we have that $|E[Z_{n,j}e^{i\lambda(U_{n} - U_{n,j})}]| \lesssim E[Z_{n,j}^{2}]^{1/2}(\psi(1,n)\varphi(c_{n}))^{1/2}$. 
We can also show that $\sum_{j=1}^{n}E[Z_{n,j}^{2}] = O(1/nh_{n})$. Then we have that 
\begin{align*}
|A_{n,3}| &\leq a_{n}^{-1/2}\sum_{j=1}^{n}|E[Z_{n,j}e^{i\lambda(U_{n} - U_{n,j})}]| \lesssim a_{n}^{-1/2}(\psi(1,n)\varphi(c_{n}))^{1/2}\sum_{j=1}^{n}E[Z_{n,j}^{2}]^{1/2}\\
&\lesssim {(\psi(1,n)\varphi(c_{n}))^{1/2} \over \sqrt{a_{n}}}n^{1/2}\left(\sum_{j=1}^{n}E[Z_{n,j}^{2}]\right)^{1/2} \lesssim (n\psi(1,n)\varphi(c_{n}))^{1/2} \to 0\ \text{as}\ n \to \infty. 
\end{align*}
\end{proof}

\begin{proof}[Proof of Theorem \ref{P2}]
We prove Theorem \ref{P2} in two steps. 

(Step1): In this step, we compute asymptotic bias and variance of the estimator. Note that we can replace $\widehat{f}$ by $f$ in the definition of $\widehat{\mu}_{b_{n}}$ since we already have the consistency of $\widehat{f}$ from Proposition \ref{P1} and define  $\widetilde{\mu}_{b_{n}}$ as $\widehat{\mu}_{b_{n}}$ with replacement of $\widehat{f}$ by $f$. We can decompose $\widetilde{\mu}_{b_{n}}(x) - \mu(x)$ as follows. 
\begin{align*}
\widetilde{\mu}_{b_{n}}(x) - \mu(x) &= {1 \over nb_{n}f(x)}\sum_{j=1}^{n}\!\sigma(X(s_{j}))V(s_{j})K\!\!\left(\!{x - X(s_{j}) \over b_{n}}\!\right) + {1 \over nb_{n}f(x)}\sum_{j=1}^{n}\!\mu(X(s_{j}))K\!\!\left(\!{x - X(s_{j}) \over b_{n}}\!\right) - \mu(x) \\
&=: V_{n}(\mu) + B_{n}(\mu). 
\end{align*}
For $B_{n}(\mu)$, by a change of variables and the dominating convergence theorem, we have that $b_{n}^{-2}E[B_{n}(\mu)] \to 0$ as $n \to \infty$. 
For $V_{n}(\mu)$, by a change of variables and the dominating convergence theorem, we have that $n b_{n}\Var(V_{n}(\mu)) \to \sigma^{2}(x)f(x)^{-1}\|K\|_{L^{2}} $as $n \to \infty$.

(Step2): The central limit theorem follows from Proposition \ref{CLT} with $G(x) = \sigma(x)$ and $H(x) = x$. The asymptotic negligibility of correlation between different design points can be shown by almost the same argument as the proof of Proposition \ref{P1}. Therefore, we obtain the desired result.  
\end{proof}

\begin{proof}[Proof of Theorem \ref{P3}]
We prove Theorem \ref{P3} in two steps. 

(Step1): In this step, we compute asymptotic bias and variance of the estimator. Note that as in the proof of Proposition \ref{P2}, we can replace $\widetilde{f}$ by $f$ in the definition of $\widehat{\sigma}_{h_{n}}^{2}$. Define  $\widetilde{\sigma}_{h_{n}}^{2}$ as $\widehat{\sigma}_{h_{n}}^{2}$ with replacement of $\widetilde{f}$ by $f$. We can decompose $\widetilde{\sigma}_{h_{n}}^{2}(x)$ as follows. 
\begin{align*}
\widetilde{\sigma}_{h_{n}}^{2}(x) &= {1 \over nh_{n}f(x)}\sum_{j=1}^{n}\{Y(s_{j}) - \mu(X(s_{j}))\}^{2}K\left({x - X(s_{j}) \over h_{n}}\right)\\
&\quad + {2 \over nh_{n}f(x)}\sum_{j=1}^{n}\left\{Y(s_{j}) - \mu(X(s_{j}))\right\}\left\{\widehat{\mu}_{b_{n}}^{\ast}(X(s_{j})) - \mu(X(s_{j}))\right\}K\left({x - X(s_{j}) \over h_{n}}\right)\\
&\quad + {1 \over nh_{n}f(x)}\sum_{j=1}^{n}\left\{\widehat{\mu}_{b_{n}}^{\ast}(X(s_{j})) - \mu(X(s_{j}))\right\}^{2}K\left({x - X(s_{j}) \over h_{n}}\right). 
\end{align*}
Since $\widehat{\mu}_{b_{n}}^{\ast}(y) - \mu(y) = O(b_{n}^{4}) + O_{P}(1/\sqrt{nb_{n}})$ for fixed $x \in \mathbb{R}$ and $y \in \{z: |z-x| \lesssim b_{n}\}$, and $b_{n} \sim h_{n}$ we have that 
\begin{align*}
\widetilde{\sigma}_{h_{n}}^{2}(x) &= {1 \over nh_{n}f(x)}\sum_{j=1}^{n}(Y(s_{j}) - \mu(X(s_{j})))^{2}K\left({x - X(s_{j}) \over h_{n}}\right) + o_{P}(1/\sqrt{nh_{n}})\\
&= {1 \over nh_{n}f(x)}\sum_{j=1}^{n}\sigma^{2}(X(s_{j}))V^{2}(s_{j})K\left({x - X(s_{j}) \over h_{n}}\right) + o_{P}(1/\sqrt{nh_{n}}).
\end{align*}
Then we can decompose $\widetilde{\sigma}_{h_{n}}^{2}(x) - \sigma^{2}(x)$ as follows. 
\begin{align*}
\widetilde{\sigma}_{h_{n}}^{2}(x) - \sigma^{2}(x) &= {1 \over nh_{n}f(x)}\sum_{j=1}^{n}\sigma^{2}(X(s_{j}))\left(V^{2}(s_{j})-1\right)K\left({x - X(s_{j}) \over h_{n}}\right)\\
&\quad + {1 \over nh_{n}f(x)}\sum_{j=1}^{n}\sigma^{2}(X(s_{j}))K\left({x - X(s_{j}) \over h_{n}}\right) - \sigma^{2}(x) + o_{P}(1/\sqrt{nh_{n}})\\
&=: V_{n}(\sigma) + B_{n}(\sigma) + o_{P}(1/\sqrt{nh_{n}}). 
\end{align*}
For $B_{n}(\sigma)$, by a change of variables and the dominating convergence theorem, we have that $h_{n}^{-2}E[B_{n}(\sigma)] \to 0$ as $n \to \infty$. 
For $V_{n}(\sigma)$, by a change of variables and the dominating convergence theorem, we have that $n h_{n}\Var(V_{n}(\sigma)) \to V_{4}\sigma^{4}(x)f(x)^{-1}\|K\|_{L^{2}}$ as $n \to \infty$. 

(Step2): The central limit theorem follows from Proposition \ref{CLT} with $G(x) = \sigma^{2}(x)$ and $H(x) = x^{2}-1$. The asymptotic negligibility of correlation between different design points can be shown by almost the same argument as the proof of Proposition \ref{P1}. Therefore, we obtain the desired result.  
\end{proof}


\begin{thebibliography}{99} 

\bibitem[Al-Sulami et al.(2017)]{AlJiLuZh17}
Al-Sulami, D., Jinag, Z., Lu, Z. and Zhu, J. (2017). Estimation for semiparametric nonlinear regression of irregularly located spatial time series data. 
\textit{Econometrics and Statistics} \textbf{2}, 22-35. 


\bibitem[Bernstein(1926)]{Be26}
Bernstein, S.N. (1926). Sur l\'extension du th\'eor\`eme limite du calcul des probabilit\'es aux sommes de quantit\'es d\'ependantes. 
\textit{Ann. Math.} \textbf{21}, 1-59. 


\bibitem[Bolthausen(1982)]{Bo82}
Bolthausen, E. (1982). On the central limit theorem for stationary mixing random fields. 
\textit{Ann. Probab.} \textbf{4}, 1047-1050. 

\bibitem[Chernozhukov, Chetverikov and Kato(2013)]{ChChKa13}
Chernozhukov, V., Chetverikov, D. and Kato, K. (2013). Inference on causal and structural
parameters using many moment inequalities. To appear in \textit{Review of Economic Studies}.

\bibitem[Fan and Yao(2003)]{FaYa03}
Fan, J. and Yao, Q. (2003). 
\textit{Nonlinear time series: Nonparametric and Parametric Methods}, Springer, New York.


\bibitem[Gao, Lu and Tj\o stheim(2006)]{GaLuTj06}
Gao, J., Lu, Z. and Tjstheim, D. (2006). Estimation in semi-parametric spatial regression. 
\textit{Ann. Statist.} \textbf{34}, 1395-1435.

\bibitem[Guyon(1987)]{Gu87}
Guyon, X. (1987). Estimation d'un champ par pseudo-vraisemblance conditionnelle: \'Etude
asymptotique et application au cas Markovien. In Spatial Processes and Spatial Time
Series Analysis. Proc. 6th Franco–Belgian Meeting of Statisticians (J.-J. Droesbeke et al.,
eds.) 15–62. FUSL, Brussels.

\bibitem[Hall and Patil(1994)]{HaPa94}
Hall, P. and Patil, P. (1994). Properties of nonparametric estimators of autocovariance for stationary random fields.  
\textit{Probab. Th. Relat. Fields} \textbf{99}, 399-424. 

\bibitem[Hallin et al.(2004)]{HaLuTr04}
Hallin, M., Lu, Z. and Tran, L. (2004). Local linear spatial regression. 
\textit{Ann. Statist.} \textbf{32}, 2469-2500.

\bibitem[Hallin et al.(2009)]{HaLuYu09}
Hallin, M., Lu, Z. and Yu, K. (2009). Local linear spatial quantile regression. 
\textit{Bernoulli} \textbf{15}, 659-686.


\bibitem[Jenish(2012)]{Je12}
Jenish, N. (2012). Nonparametric spatial regression under near-epoch dependence.
\textit{J. Econometrics} \textbf{167}, 224-239.

\bibitem[Jenish and Prucha(2009)]{JePr09}
Jenish, N. and Prucha, I. (2009). Central limit theorems and uniform laws of numbers for arrays of random fields. 
\textit{J. Econometrics} \textbf{150}, 86-98.

\bibitem[Jenish and Prucha(2012)]{JePr12}
Jenish, N. and Prucha, I. (2012). On spatial process and asymptotic inference under near-epoch dependence. 
\textit{J. Econometrics} \textbf{170}, 178-190.

\bibitem[Kurisu(2018)]{Ku18}
Kurisu, D. (2018). Nonparametric inference on L\'evy measures of L\'evy-driven Ornstein-Uhlenbeck processes under discrete observations. arXiv:1803.08671. 

\bibitem[Li(2016)]{Li16}
Li, L. (2016). Nonparametric regression on random fields with random design using wavelet method.
\textit{Stat. Inference Stoch. Process.} \textbf{19}, 51-69. 

\bibitem[Li, Lu and Linton(2012)]{LiLuLi12}
Li, D., Lu, Z. and Linton, O. (2012). Local linear fitting under near epoch dependence: uniform consistency with convergence rates. 
\textit{Econometric Theory} \textbf{28}, 935-958. 

\bibitem[Liu and Li(2009)]{LiLi09}
Liu, W. and Li, Z. (2009). Strong approximation for a class of stationary processes. 
\textit{Stochastic Process. Appl.} \textbf{119}, 249-280. 

\bibitem[Lu and Linton(2007)]{LuLi07}
Lu, Z. and Linton, O. (2007). Local linear fitting under near epoch dependence. 
\textit{Econometric Theory} \textbf{23}, 37-70. 

\bibitem[Lu et al.(2009)]{LuStTjYa09}
Lu, Z., Steinskog, D.J., Tj\o stheim, D. and Yao, Q. (2009). Adaptively varying- coefficient
spatiotemporal models. 
\textit{J. Roy. Statist. Soc. Ser. B} \textbf{71}, 859-880.

\bibitem[Lu and Tj\o stheim(2014)]{LuTj14}
Lu, Z. and Tj\o stheim, D. (2014). Nonparametric estimation of probability density functions for irregularly observed spatial data. 
\textit{J. Amer. Statist. Soc.} \textbf{109}, 1546-1564.  

\bibitem[Machkouri and Stoica(2008)]{MaSt08}
Machkouri, M.E. and Stoica, R. (2008). Asymptotic normality of kernel estimates in a regression model for random fields. 
\textit{J. Nonparametric Statist.} \textbf{22}, 955-971. 

\bibitem[Machkouri(2011)]{Ma11}
Machkouri, M.E. (2011). Asymptotic normality of the Parzen-Rosenblatt density estimator for strongly mixing random fields. 
\textit{Stat. Inference Stoch. Process.} \textit{14}, 73-84. 

\bibitem[Machkouri et al.(2013)]{MaVoWu13}
Machkouri, M.E., Voln\'y, D. and Wu, B.W. (2013). A central limit theorem for stationary random fields.
\textit{Stochastic Process Appl.} \textbf{123}, 1-14. 

\bibitem[Machkouri et al.(2017)]{MaSeOu18}
Machkouri, M.E., Es-Sebaiy, K. and Ouassou, I. (2017). On local linear regression for strongy mixing random fields. 
\textit{J. Multivaraite Anal.} \textbf{156}, 103-115. 

\bibitem[Matsuda and Yajima(2009)]{MaYa09}
Matsuda, Y. and Yajima, Y. (2009). Fourier analysis of irregularly spaced data on $\mathbb{R}^{d}$.
\textit{J. Royal Statist. Soc. Ser. B}  \textbf{71}, 191-217.




\bibitem[Robinson(2008)]{Ro08}
Robinson, P. (2008). Developments in the analysis of spatial data. 
\textit{J. Japan Statist. Soc.} (issue in honour of H. Akaike) \textbf{38}, 87-96.

\bibitem[Robinson(2011)]{Ro11}
Robinson, P. (2011). Asymptotic theory for nonparametric regression with spatial data. 
\textit{J. Econometrics} \textbf{165}, 5-19. 


\bibitem[R\"ollin(2011)]{Rol11}
R\"ollin, A. (2011). Stein's method in high dimensions with applications. 
\textit{Ann. Inst. H. Poincar\'e Probab. Statist.} \textbf{49}, 529-549. 


\bibitem[Rosenblatt(1985)]{Ro85}
Rosenblatt, M. (1985). Stationary Sequences and Random Fields. Birkh\"auser, Boston.

\bibitem[Roussas and Ioannides(1987)]{RoIo87}
Roussas, G. and Ioannides, D. (1987). Moment inequalities for mixing sequence of random variables. \textit{Stochastic Anal. Appl.} \textbf{5}, 60-120.

\bibitem[Stein(1986)]{St86}
Stein, C. (1986). Approximate computation of expectations. 
Institute of Mathematical Statistics Lecture Notes,
Monograph Series, 7. 


\bibitem[Yan et al.(2014)]{YaHoWeZu14}
Yan, S., Hongjia, Y., Wenyang, Z. and Zudi, L. (2014). A semiparametric spatial dynamic model. 
\textit{Ann. Statist.} \textbf{42}, 700-727. 

\bibitem[Zhang and Wu(2017)]{ZhWu17}
Zhang, D. and Wu, W.B. (2017). Gaussian approximation for high dimensional time series. 
\textit{Ann. Statist.} \textbf{45}, 1895-1919. 



\bibitem[Zhang and Cheng(2017)]{ZhCh17}
Zhang, X. and Cheng, G. (2017). Gaussian approximation for high dimensional vector under physical dependence. 
To appear in \textit{Bernoulli}. 

\end{thebibliography}
\end{document}